\documentclass[12pt]{amsart}
\usepackage{cases}
\usepackage{cases}
\usepackage{cases}
\usepackage{mathrsfs}
\usepackage{amsfonts}
\usepackage{amscd}
\usepackage{tikz}
\usepackage{latexsym,amssymb}
\usepackage{a4wide}
\usepackage{epic,eepic,latexsym, amssymb, amscd, amsfonts, xypic, floatflt}
\usepackage{amsthm}
\usepackage{amsmath}
\usepackage{amscd}
\usepackage{graphicx}
\usepackage[colorlinks,linkcolor=blue,citecolor=blue, pdfstartview=FitH]{hyperref}
\usepackage[mathscr]{eucal}
\usepackage{verbatim}

\input xy
\xyoption{all} \numberwithin{equation}{section}
\begin{document}

\title[Cyclotomic expansions for the colored HOMFLY-PT invariants ]
{Cyclotomic expansion for the colored HOMFLY-PT invariants of double
twist knots}
\author[Qingtao Chen, Kefeng Liu and Shengmao Zhu]{Qingtao Chen, Kefeng
Liu and Shengmao Zhu}
\address{Division of Science \\
New York University Abu Dhabi \\
Abu Dhabi \\
United Arab Emirates} \email{chenqtao@nyu.edu}
\address{Mathematical Science Research Center \\
Chongqing University of Technology  \\
Chongqing, P. R. China. }
\address{Department of mathematics \\
University of California at Los Angeles, Box 951555\\
Los Angeles, CA, 90095-1555.} \email{liu@math.ucla.edu}

\address{Department of Mathematics \\ Zhejiang International Studies
University}
\address{Center of Mathematical Sciences \\
Zhejiang University, Box 310027 \\
Hangzhou, P. R. China. } \email{szhu@zju.edu.cn}

\begin{abstract}
In this short note, we prove the cyclotomic expansion formula for
the colored HOMFLY-PT invariant of double twist knots, it confirms
the cyclotomic expansion conjecture for $SU(N)$-invariants proposed
in \cite{CLZ}.
\end{abstract}

\maketitle

\theoremstyle{plain} \newtheorem{thm}{Theorem}[section] \newtheorem{theorem}[%
thm]{Theorem} \newtheorem{lemma}[thm]{Lemma} \newtheorem{corollary}[thm]{%
Corollary} \newtheorem{proposition}[thm]{Proposition} \newtheorem{conjecture}%
[thm]{Conjecture} \theoremstyle{definition}
\newtheorem{remark}[thm]{Remark}
\newtheorem{remarks}[thm]{Remarks} \newtheorem{definition}[thm]{Definition}
\newtheorem{example}[thm]{Example}






\section{Introduction}
Colored HOMFLY-PT invariant of link is an important quantum
invariant in mathematics and physics, we refer to \cite{Zhu} and the
references therein for a recent review of the colored HOMFLY-PT
invariant and related topics. During the past several years, we
spent a lot of efforts to study the general structures for colored
HOMFLY-PT invariants. Motivated by the congruence skein relations
discovered in \cite{CLPZ} and the celebrated cyclotomic expansion
formula for colored Jones polynomial due to Habiro \cite{Hab},  we
proposed the cyclotomic expansion conjecture in \cite{CLZ} (see also
Conjecture 2.4 in version2 of \cite{NO}) for the colored $SU(n)$
invariants $J_{N}^{SU(n)}(\mathcal{K};q)$ of a knot $\mathcal{K}$
which is defined as the $a=q^n$ specialization of the normalized
colored HOMFLY-PT invariant
\begin{align} \label{formula-normalizedhomfly}
\mathscr{H}_{N}(\mathcal{K};q,a):=\frac{W_{(N)}(\mathcal{K};q,a)}{W_{(N)}(U;q,a)},
\end{align}
where $W_{(N)}(\mathcal{K};q,a)$ denotes the colored HOMFLY-PT
invariant of $\mathcal{K}$ with the symmetric representation labeled
by partition $(N)$, and $U$ denotes the unknot.
$\mathscr{H}_N(\mathcal{K};q,a)$ is normalized so that it is to be
$1$ for the unknot $U$.
\begin{conjecture} \label{conjecture-cycl}For any knot
$\mathcal{K}$, there exist Laurent polynomials
$H_k^{(n)}(\mathcal{K})\in \mathbb{Z}[q,q^{-1}]$, independent of $N$
($N\geq 0$). Such that
\begin{align}
J_{N}^{SU(n)}(\mathcal{K};q)=\sum_{k=0}^{N}C_{N+1,k}^{(n)}H_k^{(n)}(\mathcal{K}),
\end{align}
where $C_{N+1,k}^{(n)}=\{N-(k-1)\}\{N-(k-2)\}\cdots
\{N-1\}\{N\}\{N+n\}\{N+n+1\}\cdots \{N+n+(k-1)\}$, for $k=1,..,N$,
and $C_{N+1,0}^{(n)}=1$. In particular,
$J_0^{SU(n)}(\mathcal{K};q)=H_0^{(n)}(\mathcal{K})=1$.
\end{conjecture}
For a given knot or link, it is difficult to calculate its colored
HOMFLY-PT invariant in general. In \cite{Kaw1}, Kawagoe provided
some formulae for the colored HOMFLY-PT invariants
 with symmetric representations based on
the HOMFLY skein theory. Such formulae are useful to compute these
invariants of the knots and links with twisted strands of opposite
orientations. Furthermore, in the recent work \cite{Kaw2}, Kawagoe
obtained a rigorous single sum formula for the colored HOMFLY-PT
polynomial $\mathscr{H}_{N}(\mathcal{K}_p;q,a)$ of the twist knot
$\mathcal{K}_p$.

In this short note, based on Kawagoe's recent results \cite{Kaw2},
we apply the techniques in \cite{Mas} to derive the following
cyclomotic expansion formula for
$\mathscr{H}_{N}(\mathcal{K}_{p,s};q,a)$ of double twist knot
$\mathcal{K}_{p,s}$

\begin{figure}[!htb] \label{fig:doubletwist}
\begin{align*}
\includegraphics[width=120
pt]{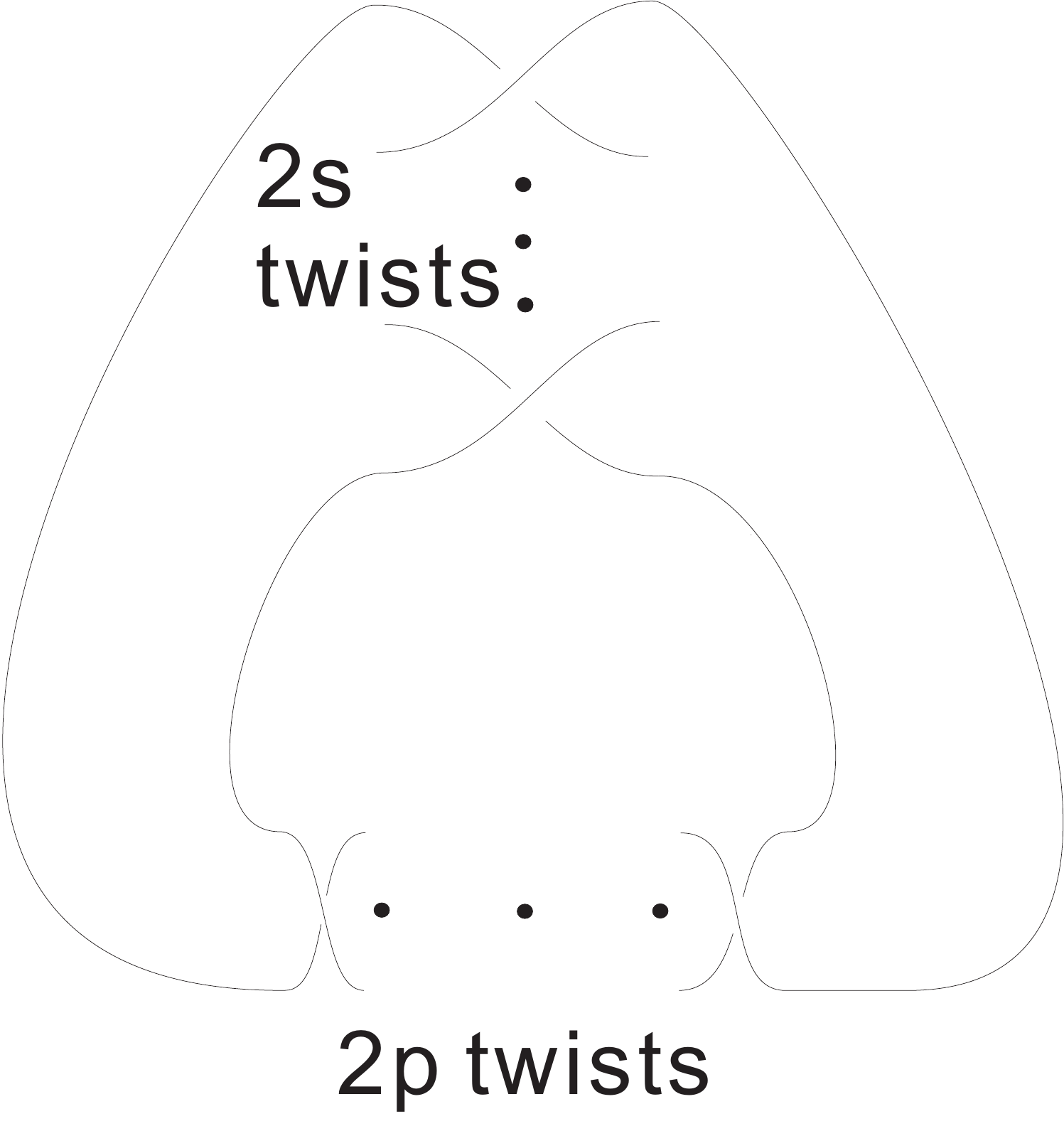}
\end{align*}
\caption{Double twist knot $\mathcal{K}_{p,s}$}
\end{figure}

\begin{theorem} \label{Theorem-main}
The colored HOMFLY-PT invariant of the double twist knot
$\mathcal{K}_{p,s}$ is given by
\begin{align} \label{formula-HN}
\mathscr{H}_{N}(\mathcal{K}_{p,s};q,a)=\sum_{k=0}^{N}(-1)^k(a^kq^{k(k-1)})^{2(p+s)}\tilde{C}_{k,k}^{(p)}\tilde{C}_{k,k}^{(s)}\left[
\begin{array}{@{\,}c@{\,}}N \\ k
\end{array} \right] \{ N+k-1;a \}_k \{k-2;a \}_k,
\end{align}
where the coefficient $\tilde{C}_{k,k}^{(p)}\in
\mathbb{Z}[q,q^{-1}]$, whose explicit expression is given by formula
(\ref{formula-tildec}).
\end{theorem}
In particular, let $a=q^n$ in the above formula (\ref{formula-HN}),
we obtain
\begin{corollary}
The Conjecture \ref{conjecture-cycl} holds for the $SU(n)$-invariant
$J_{N}^{SU(n)}(\mathcal{K}_{p,s};q)$ of the double twist knot
$\mathcal{K}_{p,s}$.
\end{corollary}

\section{Preliminaries}
In this section, we briefly review Kawagoe's recent work \cite{Kaw2}
and fix the notations. Let $a$ and $q$ be two non-zero variables in
$\mathbb{C}$. For an integer $n$, we define the symbols by
\begin{align*}
[n] &= \frac{q^n-q^{-n}}{q-q^{-1}}, \qquad \{ n \} =  q^{n} -
q^{-n}, \qquad \{ n ; a \} =  aq^{n}-a^{-1}q^{-n}.
\end{align*}
For integers $n >0 , i \geq 0$, we introduce the products of  $i$
terms of these symbols by
\begin{align*}
[n]_{i} &= [n] [n-1] \cdots [n-i+1], \\
\{ n \}_{i} &=  \{ n \} \{ n-1 \} \cdots \{ n-i+1 \},  \\
\{ n ; a \}_{i} &=  \{ n; a \} \{ n-1; a \} \cdots \{ n-i+1; a \},  \\
\{ -n ; a \}_{i} &=  \{ -n; a \} \{ -n+1; a \} \cdots \{ -n+i-1; a
\},
\end{align*}
which are defined to be $1$ if $i = 0$. Furthermore, we let
\begin{align*}
[n]! =[n]_{n}, \quad \{ n \} ! =  \{ n \}_{n}, \quad \left[
\begin{array}{@{\,}c@{\,}}n \\ i \end{array} \right] =
\frac{[n]!}{[i]! [n-i]!}.
\end{align*}

The HOMFLY skein module $\mathcal{S}(M)$ of a oriented 3-manifold
$M$ is the free $\mathbb{C}$-module modulo the submodule generated
by the following HOMFLY-PT skein relations:

\begin{enumerate}
\setlength{\itemsep}{3mm}
\item  $L\cup U = \displaystyle \frac{ \{0 ; a \} }{ \{1\} } L$,
 and $\varnothing = 1$, 
\item \begin{minipage}{15pt}
        \begin{picture}(15,15) 
            \qbezier(0,15)(0,15)(15,0)
            \qbezier(0,0)(0,0)(6,6)
            \qbezier(9,9)(9,9)(15,15)
            \put(15,0){\vector(1,-1){0}}
            \put(15,15){\vector(1,1){0}}
        \end{picture}
\end{minipage}
$\, - \,$
\begin{minipage}{15pt}
        \begin{picture}(15,15) 
            \qbezier(0,0)(0,0)(15,15)
            \qbezier(0,15)(0,15)(6,9)
            \qbezier(9,6)(9,6)(15,0)
            \put(15,0){\vector(1,-1){0}}
            \put(15,15){\vector(1,1){0}}
        \end{picture}
\end{minipage}
$\, = (q-q^{-1}) \,$
\begin{minipage}{30pt}
        \begin{picture}(15,15) 
            \qbezier(0,15)(7.5,7.5)(15,15)
            \qbezier(0,0)(7.5,7.5)(15,0)
            \put(15,0){\vector(3,-2){0}}
            \put(15,15){\vector(3,2){0}}
        \end{picture}
\end{minipage}\hspace*{-3mm}, 
\item \begin{minipage}{18pt}
        \begin{picture}(18,15) 
            \qbezier(0,11.2)(10.5,9)(10.5,4.5)
            \qbezier(10.5,4.5)(10.5,0)(7.5,0)
            \qbezier(7.5,0)(4.5,0)(4.5,4.5)
            \qbezier(4.5,4.5)(4.5,5.2)(6,7.5)
            \qbezier(9,9.7)(10.5,11.2)(15,11.2)
            \put(18,12.2){\vector(4,1){0}}
        \end{picture}
\end{minipage}
$\, = a \,$
\begin{minipage}{18pt}
        \begin{picture}(15,15) 
            \qbezier(0,7.5)(0,7.5)(15,7.5)
            \put(16.5,7.5){\vector(1,0){0}}
        \end{picture}
\end{minipage}
$\, , \quad$
\begin{minipage}{18pt}
        \begin{picture}(18,15) 
            \qbezier(0,11.2)(4.5,11.2)(6,9.7)
            \qbezier(9,7.5)(10.5,5.2)(10.5,4.5)
            \qbezier(10.5,4.5)(10.5,0)(7.5,0)
            \qbezier(4.5,4.5)(4.5,0)(7.5,0)
            \qbezier(4.5,4.5)(4.5,9)(15,11.2)
            \put(18,12.2){\vector(4,1){0}}
        \end{picture}
\end{minipage}
$\; = a^{-1} \,$
\begin{minipage}{18pt}
        \begin{picture}(15,15) 
            \qbezier(0,7.5)(0,7.5)(15,7.5)
            \put(16.5,7.5){\vector(1,0){0}}
        \end{picture}
\end{minipage}$\, .$
\end{enumerate}

For an integer $n \geq 1$, we recursively define the $n$-th
$q$-symmetrizer by Figure \ref{fig:qsym}, where the $q$-symmetrizer
is denoted by a white rectangle.  The integer $n$ beside an arc
means $n$ copies of the arc.

\begin{figure} \label{fig:qsym}
\begin{align*}
\raisebox{-14pt}{
\includegraphics[width=28 pt]{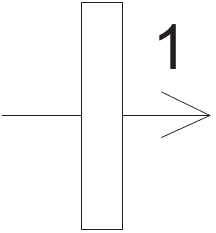}}=\raisebox{-1pt}{
\includegraphics[width=28 pt]{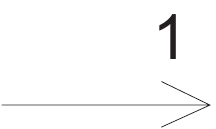}}
\end{align*}
\begin{align*}
\raisebox{-14pt}{
\includegraphics[width=28 pt]{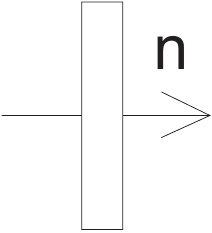}}=\frac{q^{-n+1}}{[n]}\raisebox{-14pt}{
\includegraphics[width=20 pt]{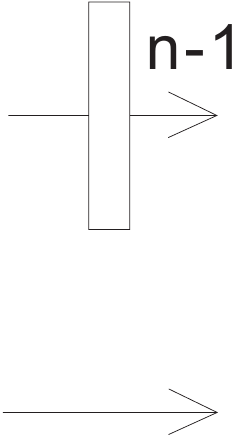}}+\frac{q^{[n-1]}}{[n]}\raisebox{-20pt}{
\includegraphics[width=80 pt]{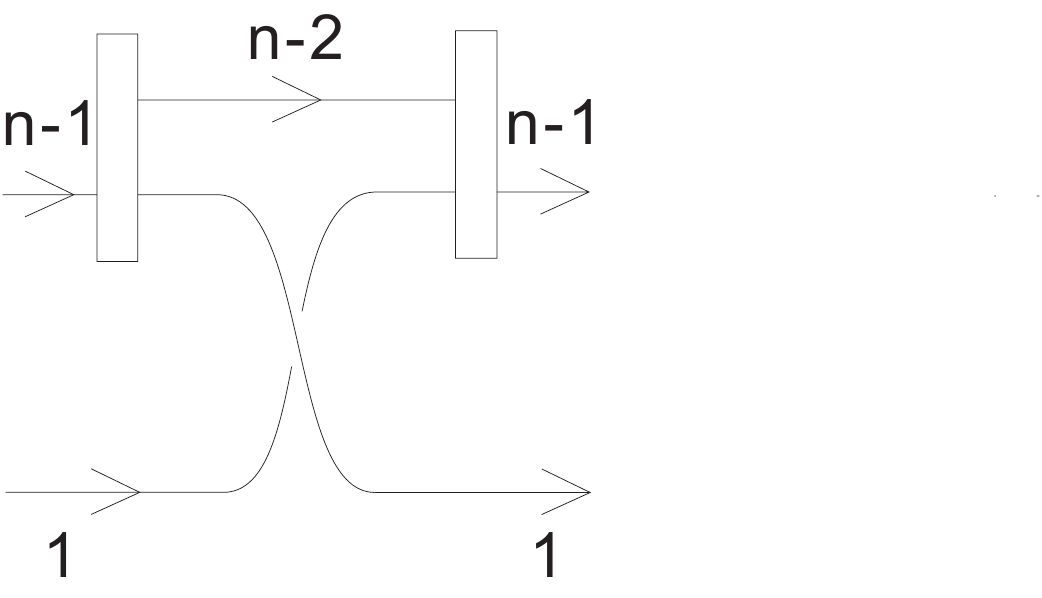}}
\end{align*}
\caption{$n$-th $q$-symmetrizer}
\end{figure}

Using the $m$-th and $n$-th $q$-symmetrizers, Kawagoe introduced the
$(m,n)$-th $q$-symmetrizer  by Figure 3, where $x_{m,n}^i$ is given
by
\begin{align*}
x_{m,n}^i = (-1)^i \left[ \begin{array}{@{\,}c@{\,}}m \\ i
\end{array} \right] \left[ \begin{array}{@{\,}c@{\,}}n \\ i
\end{array} \right] \frac{ \{ i \}!  }{  \{ m+n -2;a \}_i  }.
\end{align*}
\begin{figure} \label{fig:mnqsym}
\begin{align*}
\raisebox{-25pt}{
\includegraphics[width=28 pt]{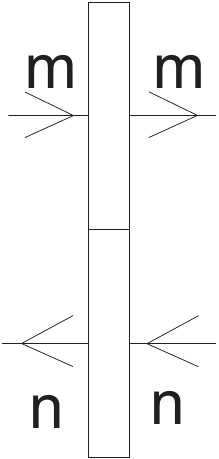}}=\sum_{i=0}^{\min{\{m,n\}}}
x_{m,n}^i\raisebox{-25pt}{\includegraphics[width=55
pt]{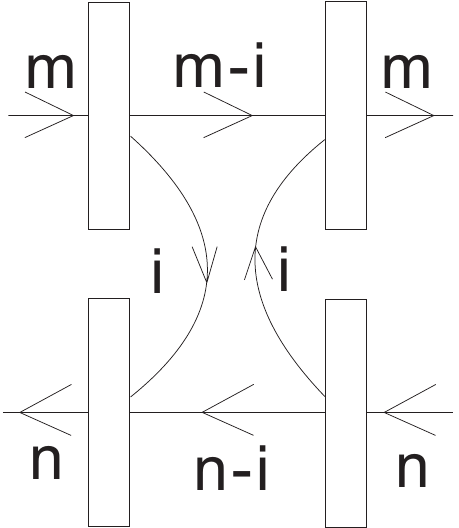}},
\end{align*}
\caption{$(m,n)$-th $q$-symmetrizer}
\end{figure}

According to their definitions, the $n$-th $q$-symmetrizer and
$(m,n)$-th $q$-symmetrizer carry the following properties

\begin{align} \label{formula-crossing}
\raisebox{-14pt}{
\includegraphics[width=58 pt]{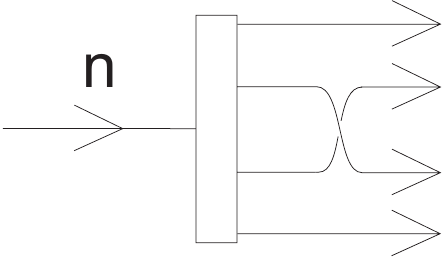}}=\raisebox{-14pt}{
\includegraphics[width=58 pt]{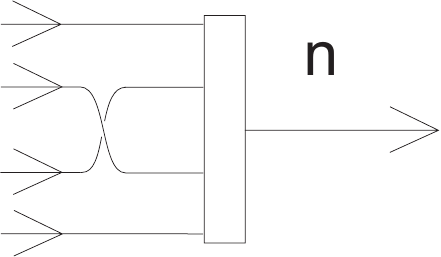}}=q\raisebox{-14pt}{
\includegraphics[width=30 pt]{oneboxn.pdf}}
\end{align}
\begin{align} \label{formula-vanishing}
\raisebox{-12pt}{
\includegraphics[width=28 pt]{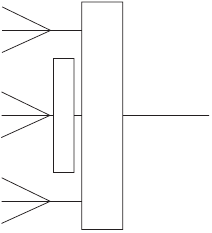}}=\raisebox{-12pt}{
\includegraphics[width=28 pt]{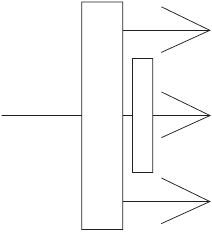}}=\raisebox{-12pt}{
\includegraphics[width=28 pt]{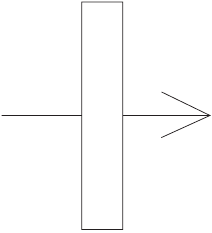}}
\end{align}

\begin{align} \label{formula-mnidempotent}
\raisebox{-14pt}{
\includegraphics[width=32 pt]{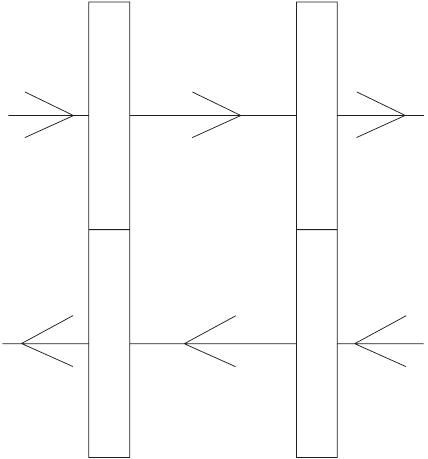}}=\raisebox{-14pt}{
\includegraphics[width=16 pt]{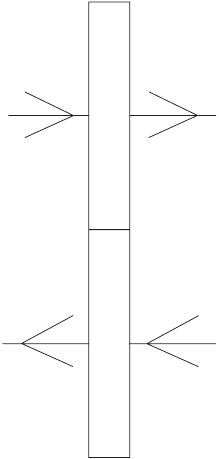}}
\end{align}

\begin{align} \label{formula-mnvanishing}
\raisebox{-14pt}{
\includegraphics[width=18 pt]{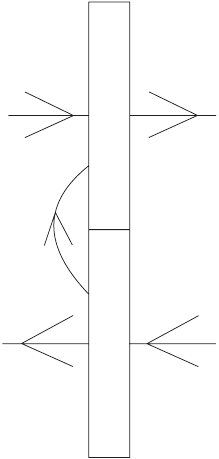}}=\raisebox{-14pt}{
\includegraphics[width=18 pt]{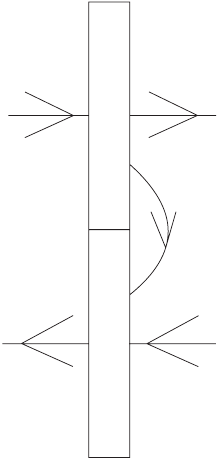}}=0
\end{align}

We denote the skein module of the solid torus $S^1 \times D^2$  by
$\mathcal{S}$. For a circle along $S^1$ of the solid torus, let $H_n
\in \mathcal{S}$ be the $n$ copies of the circles inserted by the
$n$-th $q$-symmetrizer, and $H_{n,n}$ be two copies of $H_n$, one is
anticlockwise and the other is clockwise. Let $D_{m,n}$ by $m$
copies of the anticlockwise circle and $n$ copies of the clockwise
inserted by the $(m,n)$-th $q$-symmetrizer. Kawagoe introduced two
submodule of $\mathcal{S}$,  $\mathcal{H}_{n,n}$ and
$\mathcal{D}_{n,n}$, which are spanned by $H_{i,i}$ and ${D}_{i,i}$
for $i=0,\ldots,n$, respectively. Then, he proved that
$\mathcal{D}_{n,n} = \mathcal{H}_{n,n}$. Let $\langle \; \; \rangle$
be the linear map on $\mathcal{S}$  defined by  evaluating it in
$S^3$.

As in \cite{Mas}, Kawagoe introduced the twist map $t : \mathcal{S}
\to \mathcal{S}$ induced by one right-handed twist on the solid
torus as shown in the right-hand side of Figure 4, where the twist
is acting on the bottom of the solid torus. Similarly, let $t^{-1}$
be the twist map induced by one left-hand twist. For $x \in
\mathcal{S}$, let $e_{x} : \mathcal{S} \to \mathcal{S}$ be the
encircling map which encircles an element of $\mathcal{S}$ by $x$ as
shown in the left-hand side of Figure 4, where $x$ slides into the
bottom of the solid torus and encircles the element.

\begin{figure}[!htb]\label{fig:fulltwist}
\begin{align}
\raisebox{-10pt}{
\includegraphics[width=80 pt]{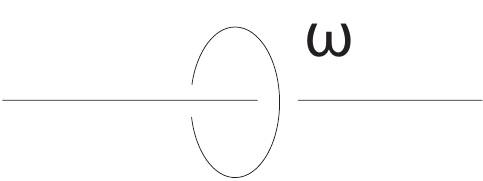}}=\raisebox{-5pt}{
\includegraphics[width=80 pt]{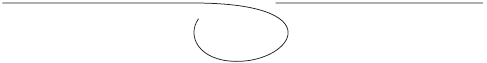}}
\end{align}
\caption{The encircling map $e_{\omega}$ and a positive full twist}
\end{figure}

We need the following lemmas which are borrowed from \cite{Kaw2}
directly.
\begin{lemma}[cf. Lemma 2.6 in \cite{Kaw2} ]  \label{lemma:Dmn}
For positive integers $m \geq n $, we have the following:
\begin{align*}
\langle D_{m,n} \rangle =\frac{ \{ m + n-1; a \} \{ m-2;a \}_{m-1}
\{ n-2;a \}_{n-1} \{-1;a\}  }{   \{ m \}! \{ n \}! }.
\end{align*}
\end{lemma}

For $i=0,\ldots,n$, Kawagoe recursively defined elements $R_i \in
\mathcal{S}$ as follows:
\begin{align*}
R_0 &= H_0 = 1, \\
R_n &= H_n - \sum_{i=0}^{n-1} \frac{ \{ n-1+i;a \}_{n-i} }{ \{
n-i\}! } R_i, \ \text{for} \ n\geq 1.
\end{align*}
Conversely, each $H_i$ can be expressed by
\begin{align*}
H_i  =  \sum_{j=0}^{i} \frac{ \{ i-1+j;a \}_{i-j} }{ \{ i-j\}! } R_j.  
\end{align*}
These elements $R_i$ play important role in calculating the colored
HOMFLY-PY polynomial of double twist knot. Actually, the behavior of
$H_n$ concerning the encircling map is very complicated, so Kawagoe
introduced $R_n$ as a linear combination of $H_i (i=0,\ldots, n)$
which simplifies the computations.

\begin{lemma}[cf. Proposition 3.4 and Corollary 3.5 in \cite{Kaw2}]\label{Coro:vanishi}
For an integer $i\geq 0$, we have
$e_{R_i}(D_{n,n})=\theta_{n,i}D_{n,n}$, where
$\theta_{n,i}=\{n\}_{i}=\{n\}_{i}\{n_i-2,a\}_{i}$, and for $i>n$,
$e_{R_i}(D_{n,n})=0$.
\end{lemma}

\begin{lemma}[cf. Lemma 2.1 in \cite{Kaw2}]\label{lemma:alpha}
We have the following formula
\begin{align} \label{formula-mn2mn}
\raisebox{-25pt}{
\includegraphics[width=100
pt]{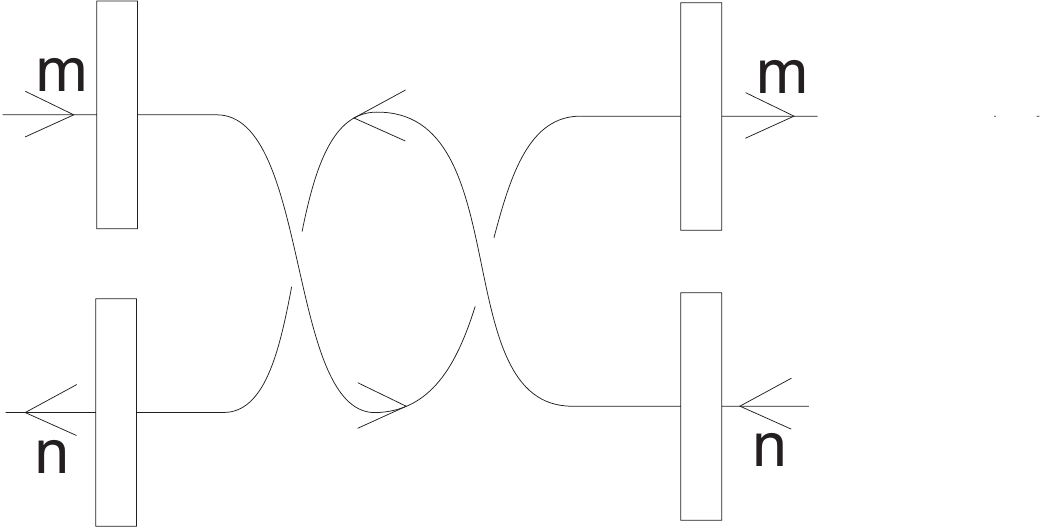}}=\sum_{i=0}^{\min{\{m,n\}}}\alpha_{m,n}^i(q,a)
\raisebox{-25pt}{\includegraphics[width=50 pt]{mnmnii.pdf}},
\end{align}
where
\begin{align} \label{formula:alpha}
 \alpha_{m,n}^i(a,q) =
(-a)^{-i} q^{-i (m+n) +\frac{i(i+3)}{2}} \frac{  \{ m \}_i \{ n \}_i
}{   \{ i \} !        }.
\end{align}
\end{lemma}
\begin{remark}
In formula (\ref{formula-mn2mn}), when all the negative crossings
change into positive crossings, the coefficients
$\alpha_{m,n}^i(q,a)$ will change into
$\alpha_{m,n}^i(q^{-1},a^{-1})$.
\end{remark}

\begin{lemma}[cf. Lemma 4.1 in \cite{Kaw2}]\label{lemma:ymn} The following holds:
\begin{align*}
\raisebox{-20pt}{
\includegraphics[width=18
pt]{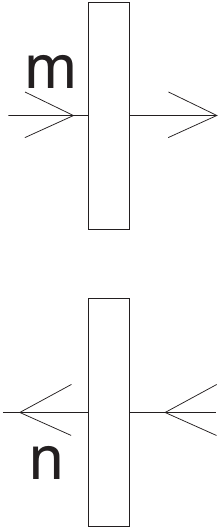}}=\sum_{i=0}^{\min{\{m,n\}}}y_{m,n}^i
\raisebox{-22pt}{\includegraphics[width=50
pt]{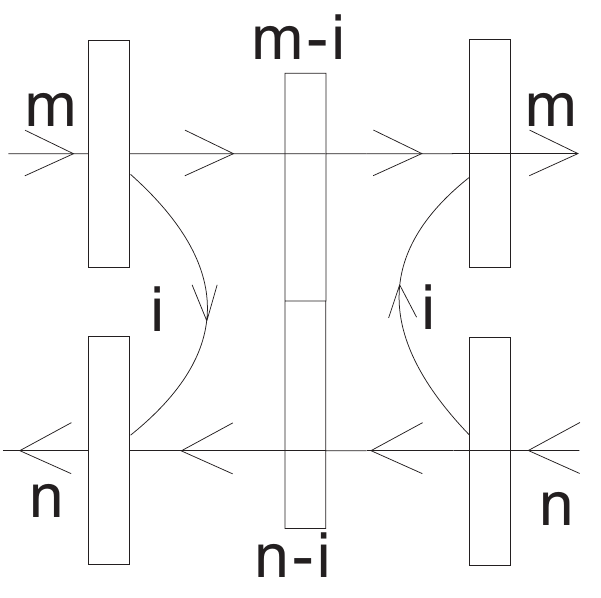}}
\end{align*}
\noindent where $y_{m,n}^i$ is defined by
\begin{align*}
y^i_{m,n} = \frac{  \{m\}_i \{n\}_i  }{ \{ i \}! \{m+n-i-1;a \}_i
}.
\end{align*}
\end{lemma}

We set the element  $\omega_n^p \in \mathcal{S}$ by
\begin{align} \label{formula-omegap}
\omega_n^p = \sum_{i=0}^n t_{i,p} R_i,
\end{align}
for some $t_{i,p} \in \mathbb{C}$, Kawagoe determined the
coefficients $t_{i,p}$ so that $e_{\omega_n^p}(x) = t^p(x)$ for $x
\in \mathcal{D}_{n,n}$ as follows.

By the idempotent and vanishing properties of $(n,n)$-the
$q$-symmetrizer, the left-hand side of Figure 5 is transformed into
$D_{n,n}$ with the multiplication of some $T_{n,p}\in\mathbb{C}$.

\begin{figure}[!htb] \label{fig:twist}
\begin{align*}
\raisebox{-55pt}{
\includegraphics[width=95 pt]{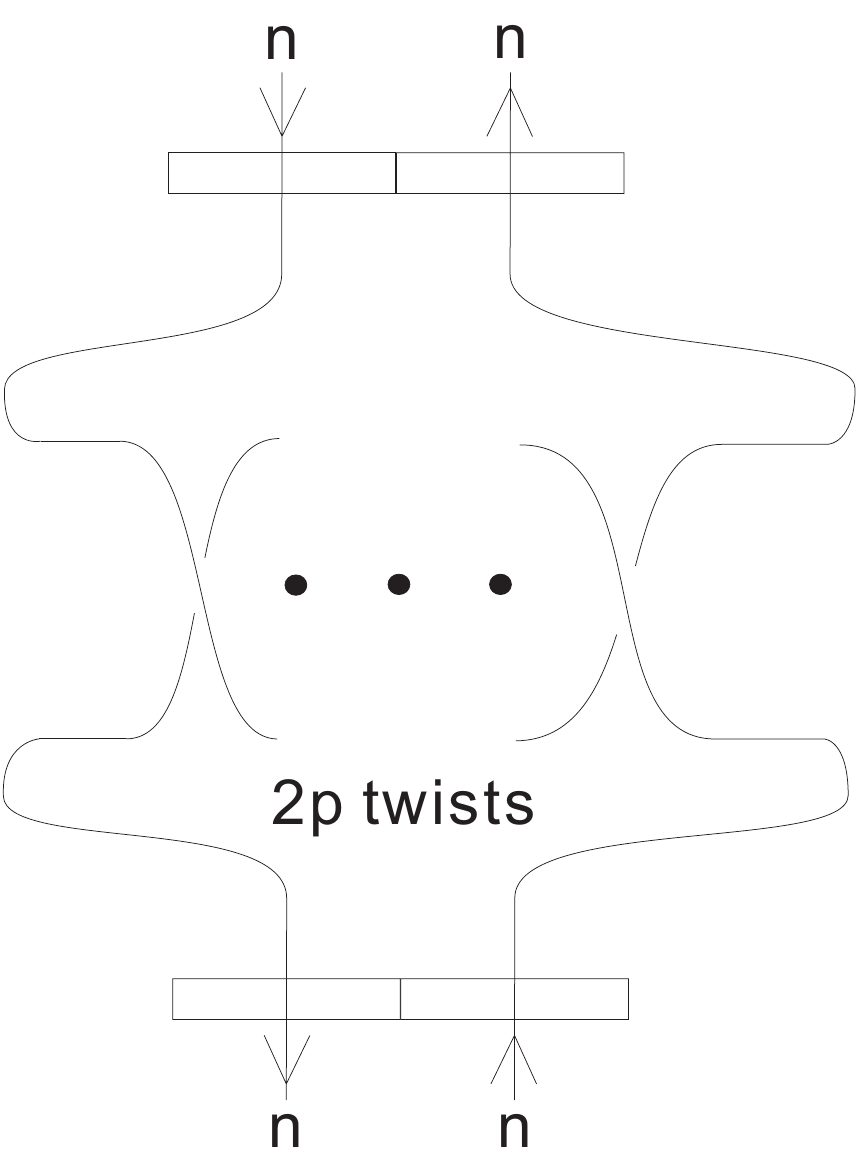}}=
T_{n,p}\quad\raisebox{-15pt}{\includegraphics[width=55
pt]{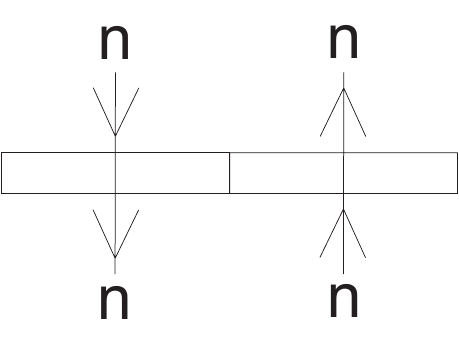}}
\end{align*}
\caption{}
\end{figure}
To determine $t_{n,p}$, Kawagoe calculated $T_{n,p}$ in two ways as
in \cite{Mas}. One is to make use of the $(n-i,n-i)$-th
$q$-symmetrizer ($i=0,\ldots,n$), and the other is to encircle $p$
full twists by $\omega_{n}^{p}$. Finally, he obtained
\begin{align}
T_{n,p} = (a^{-n}q^{-n(n-1)})^{2p} \{ n \}! \sum_{i=0}^n (-1)^i
(a^{i}q^{i(i-1)})^{2p}\frac{  \{n\}_{n-i}  }{ \{n-i\}! } \frac{
\{2i-1;a \}  }{  \{ n+i-1;a \}_{n+1}   },
\end{align}
 and
\begin{align}  \label{formula:Tnp=tnp}
T_{n,p}=
 (-1)^n  (a^nq^{n(n-1)})^{-2p} (\{n\}!)^2 t_{n,p}.
\end{align}
Hence,
\begin{align}
t_{n,p} = (-1)^n \sum_{i=0}^n   (-1)^i   \frac{( a^i
q^{i(i-1)})^{2p} }{ \{ i \}! \{ n-i \}!}. \frac{\{ 2i-1 ; a \}}{ \{
n+i-1; a \}_{n+1} }.
\end{align}

\section{Proof of the Theorem \ref{Theorem-main}}
The crucial step to prove the Theorem \ref{Theorem-main} is to give
 another explicit formula for $T_{n,p}$ which shows that $T_{n,p}$
 is a Laurent polynomial divisible by $\{n\}_{k}$.

First, we have the following lemma.

\begin{lemma} \label{lemma:2ptwist1}
We have the expansion formula
\begin{align} \label{formula-expansion}
\raisebox{-25pt}{
\includegraphics[width=120 pt]{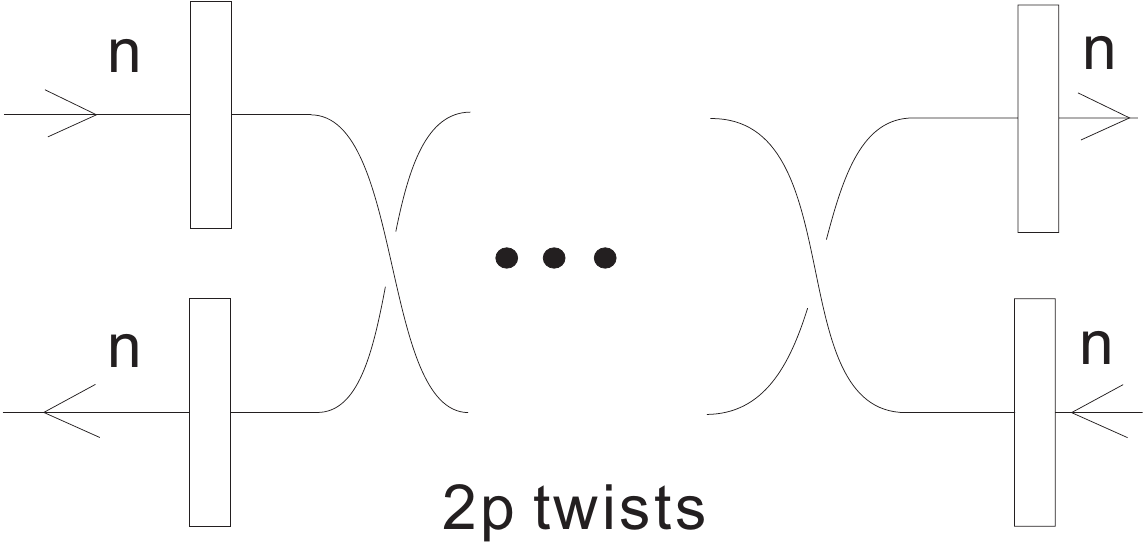}}=\sum_{k=0}^nC_{n,k}^{(p)}\raisebox{-25pt}{\includegraphics[width=60
pt]{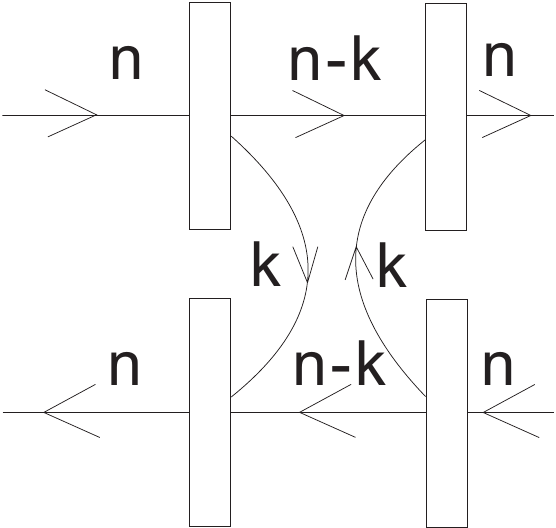}}
\end{align}
where the coefficient $C_{n,k}^{(p)}$ is a Laurent polynomial
divisible by $\{n\}_{k}$.
\end{lemma}
\begin{proof}
After applying Lemma \ref{lemma:alpha}, one term index by $k$ in
sigma sum reads
\begin{align} \label{formula-2fulltwist}
\raisebox{-20pt}{
\includegraphics[width=110 pt]{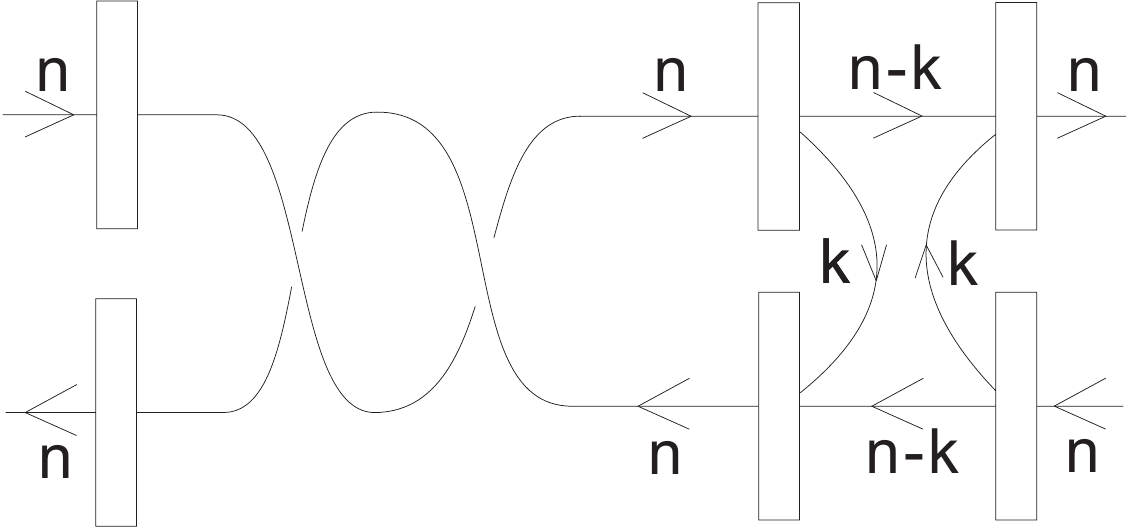}}&=\raisebox{-20pt}{
\includegraphics[width=100 pt]{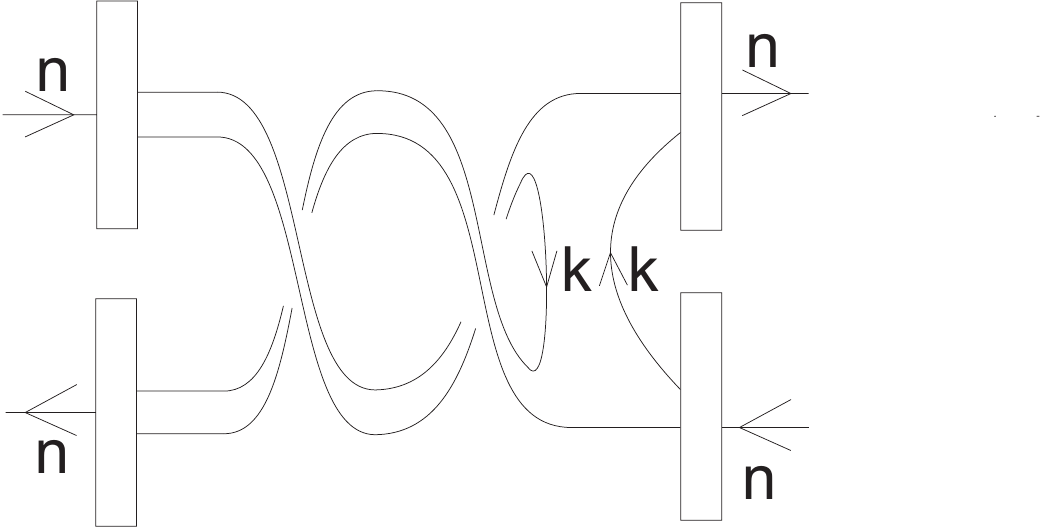}}\\\nonumber
&=(a^kq^{k(k-1)})^{-2}\raisebox{-20pt}{\includegraphics[width=100
pt]{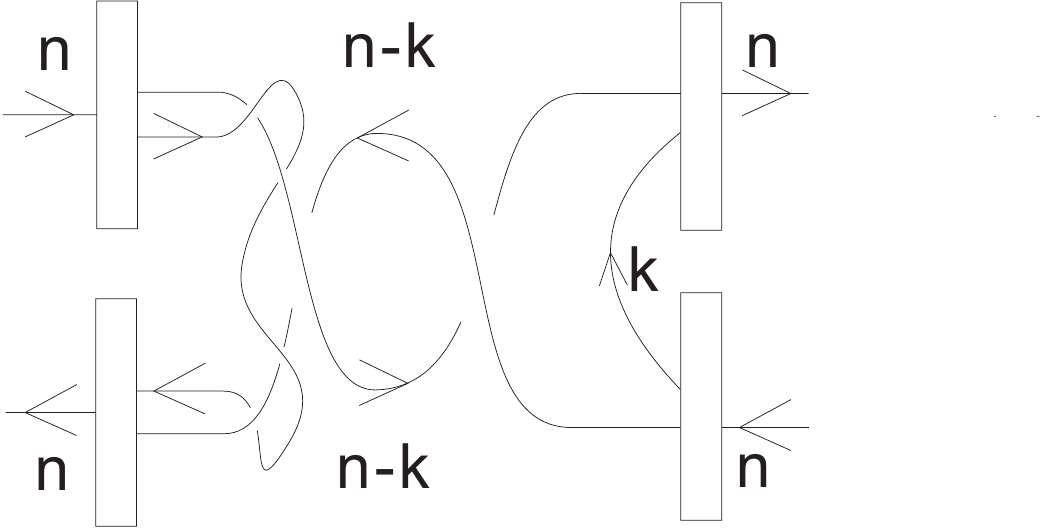}}\\\nonumber
&=(a^kq^{k(k-1)})^{-2}q^{-4k(n-k)}\raisebox{-20pt}{\includegraphics[width=100
pt]{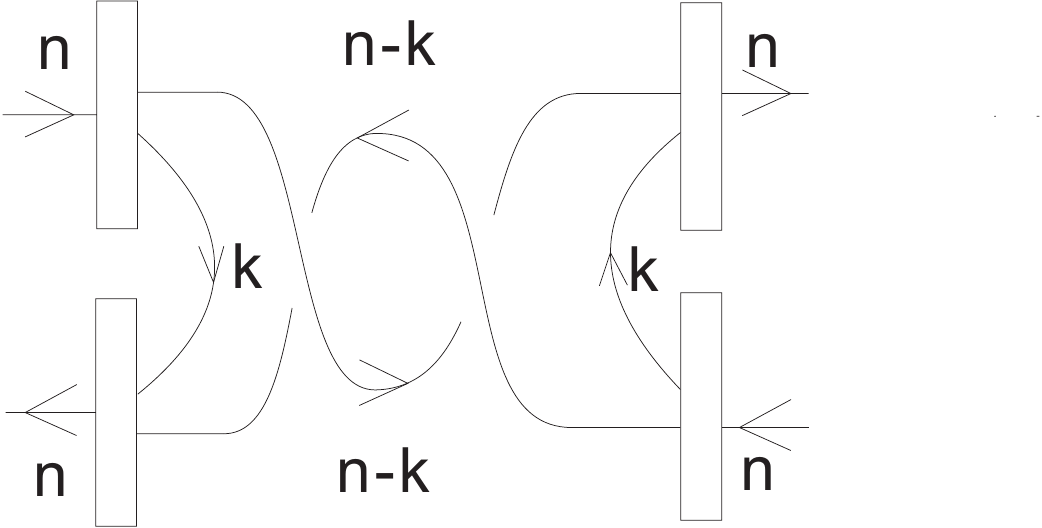}}
\end{align}
where we have used the properties (\ref{formula-crossing}) and
(\ref{formula-vanishing}). Then, applying Lemma \ref{lemma:alpha} to
the $n-k$ twisted strands in the last diagram of formula
(\ref{formula-2fulltwist}), we obtain the expansion
(\ref{formula-expansion}) for $p=2$, and
\begin{align}
C_{n,k}^{(2)}=\sum_{l_1+l_2=k}\alpha_{n,n}^{l_1}a^{-2l_1}q^{2l_1(l_1-2n+1)}\alpha_{n-l_1,n-l_1}^{l_2}.
\end{align}
The formula (\ref{formula:alpha}) for $\alpha_{n,n}^{i}$ implies
that $C_{n,k}^{(2)}$ is a Laurent polynomial divisible by
$\{n\}_{k}$. The proof for general $p$ is similar.

\end{proof}

\begin{remark} \label{remark:-p}
In formula (\ref{lemma:2ptwist1}), when all the negative crossings
change into positive crossings, the coefficients
 also change so that $a,q$ are replaced by $a^{-1},q^{-1}$, i.e.
\begin{align}
C_{n,k}^{(p)}(q,a)=C_{n,k}^{(-p)}(q^{-1},a^{-1}), \ \text{for} \
p<0.
\end{align}
\end{remark}

\begin{lemma} \label{lemma:T=C}
The following holds:
\begin{align}
T_{n,p}=C_{n,n}^{(p)}.
\end{align}
\end{lemma}
\begin{proof}
By the idempotent and vanishing properties
(\ref{formula-mnidempotent}) and (\ref{formula-mnvanishing}) for the
$(n,n)$-th $q$-symmetrizer, we have
\begin{align} \label{formula-Cnn}
\raisebox{-40pt}{
\includegraphics[width=85
pt]{nn2pnnrotate.pdf}}=\sum_{k=0}^nC_{n,k}^{(p)}\quad\raisebox{-40pt}{\includegraphics[width=75
pt]{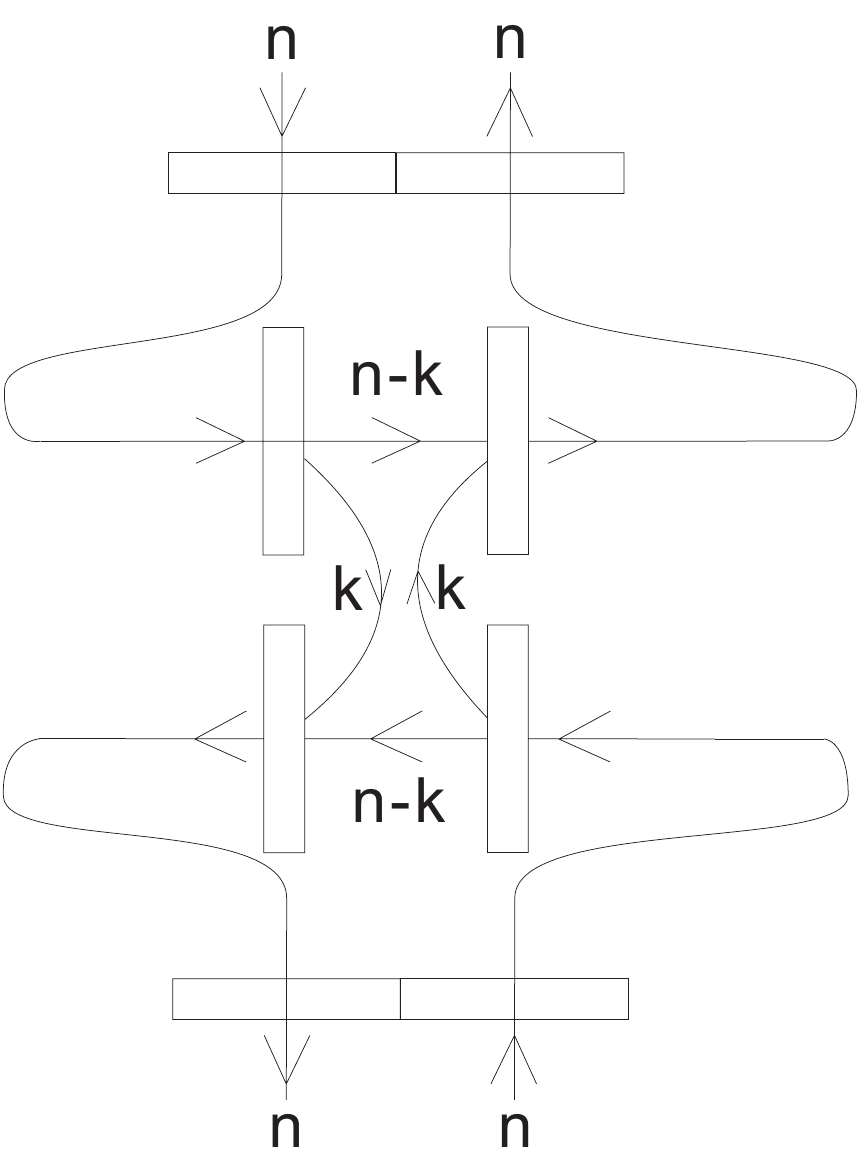}}=C_{n,n}^{(p)}\quad\raisebox{-20pt}{\includegraphics[width=60
pt]{nnnnrotate.pdf}}.
\end{align}
Comparing to the Figure 5, we obtain $T_{n,p}=C_{n,n}^{(p)}.$
\end{proof}

Actually, we can present the explicit formula for $C_{n,n}^{(p)}$
 as follows. Note that the formula (\ref{formula:alpha})
gives
\begin{align}
\alpha_{n,n}^{i}=(-a)^{-i}q^{-2ni+\frac{i(i+3)}{2}}\left[
\begin{array}{@{\,}c@{\,}}n \\ i
\end{array} \right]\{n\}_i.
\end{align}

When $p\geq 1$, we consider the sum over all multi-indices denoted
by $\underline{l}=(l_1,...,l_p)$ such that $l_i\geq 0$ for all $i$,
and $\sum_{i=1}^p l_i=n$. We put $s_i=l_1+\cdots +l_i$ for
$i=1,...,p$ and define
\begin{align}
\varphi(\underline{l})=\sum_{i=1}^{p-1}2s_i(s_i-2n+1)+\sum_{i=1}^p\frac{l_i(l_i+3)}{2}+2\sum_{i=1}^{p-1}s_il_{i+1},
\end{align}
and
\begin{align}
\left[
\begin{array}{@{\,}c@{\,}}n \\ \underline{l}
\end{array} \right]=\frac{[n]!}{[l_1]!\cdots[l_p]!}.
\end{align}
Then
\begin{align}
C_{n,n}^{(p)}&=\sum_{\underline{l}}\alpha_{n,n}^{l_1}a^{-2l_1}q^{2l_1(l_1-2n+1)}\alpha_{n-l_1,n-l_1}^{l_2}a^{-2s_2}q^{2s_2(s_2-2n+1)}\cdots\\\nonumber
&\times\alpha_{n-s_{p-2},n-s_{p-2}}^{l_{p-1}}a^{-2s_{p-1}}q^{2s_{p-1}(s_{p-1}-2n+1)}\alpha_{n-s_{p-1},n-s_{p-1}}^{l_p}\\\nonumber
&=\sum_{\underline{l}}a^{-2\sum_{i=1}^{p-1}s_i}q^{\sum_{i=1}^{p-1}2s_i(s_i-2n+1)}\alpha_{n,n}^{l_1}\alpha_{n-s_1,n-s_1}^{l_2}\cdots
\alpha_{n-s_{p-1},n-s_{p-1}}^{l_p}\\\nonumber
&=(-1)^na^{-n}q^{-2n^2}\{n\}!\sum_{\underline{l}}a^{-2\sum_{i=1}^{p-1}(p-i)l_i}q^{\varphi(\underline{l})}\left[
\begin{array}{@{\,}c@{\,}}n \\ \underline{l}
\end{array} \right]
\end{align}

For $p\leq -1$, by Remark \ref{remark:-p}, we can define
\begin{align}
C_{n,n}^{(p)}(q,a)=C^{(-p)}_{n,n}(q^{-1},a^{-1}).
\end{align}

Finally, we put
\begin{align} \label{formula-tildec}
\tilde{C}_{k,k}^{(p)}:=\frac{C_{k,k}^{(p)}}{\{k\}!}=(-1)^ka^{-k}q^{-2k^2}\sum_{\underline{l}}a^{-2\sum_{i=1}^{p-1}(p-i)l_i}q^{\varphi(\underline{l})}\left[
\begin{array}{@{\,}c@{\,}}k \\ \underline{l}
\end{array} \right],
\end{align}
then $\tilde{C}_{k,k}^{(p)}$ is Laurent polynomial in
$\mathbb{Z}[q,q^{-1}]$.

For examples,
\begin{align}
\tilde{C}_{k,k}^{(1)}&=(-1)^ka^{-k}q^{-\frac{3k(k-1)}{2}}, \ \
\tilde{C}_{k,k}^{(-1)}=a^{k}q^{\frac{3k(k-1)}{2}}, \\ \nonumber
\tilde{C}_{k,k}^{(2)}&=(-1)^ka^{-k}q^{-\frac{3k(k-1)}{2}}\sum_{l=0}^ka^{-2l}q^{-3kl+l(l+2)}\left[
\begin{array}{@{\,}c@{\,}}k \\ l
\end{array} \right], \\ \nonumber
\tilde{C}_{k,k}^{(-2)}&=a^{k}q^{\frac{3k(k-1)}{2}}\sum_{l=0}^ka^{2l}q^{3kl-l(l+2)}\left[
\begin{array}{@{\,}c@{\,}}k \\ l
\end{array} \right].
\end{align}

Now, we are ready to calculate the colored HOMFLY-PT polynomial of
the double twist knots $\mathcal{K}_{p,s}$. By its definition, the
$N$-th colored HOMFLY-PT polynomial for a knot $\mathcal{K}$ is
given by
\begin{align}
\mathscr{H}_N(\mathcal{K}) = \frac{\langle \mathcal{K}(H_N) \rangle
}{ \langle U(H_N) \rangle } = \frac{ \{N\}!  }  {\{ N-1;a\}_N}
{\langle \mathcal{K}(H_N) \rangle},
\end{align}
where $\mathcal{K}(H_N)$  denotes the knot $\mathcal{K}$ cabled by
$H_N$ with compatible orientations. Note that if we assume that the
framing of the knot $\mathcal{K}$ is to be $0$, then this definition
of $\mathscr{H}_N(\mathcal{K})$ is equal to the formula
$(\ref{formula-normalizedhomfly})$.

For two integers $p,s$, the double twist knot $\mathcal{K}_{p,s}$ is
described in Figure \ref{fig:doubletwist}. Note that
$\mathcal{K}_{-p,-s}$ is the mirror of $\mathcal{K}_{p,s}$, we will
assume $p\in \mathbb{Z}$ and $s\geq 1$.  From this, in particular,
we can see that $\mathcal{K}_{1,1}$ is a left-handed trefoil,
$\mathcal{K}_{-1,1}$ is a figure-eight knot and $\mathcal{K}_{2,1}$
is the $5_2$ knot.

We observe that the techniques used in \cite{Mas,Kaw2} are also
workable to compute the double twist knot $\mathcal{K}_{p,s}$.
First, by using the definition of the encircling map and
$\omega^{p}$, we have the following description for double twist
knot with framing zero.

\begin{align*}
\includegraphics[width=120
pt]{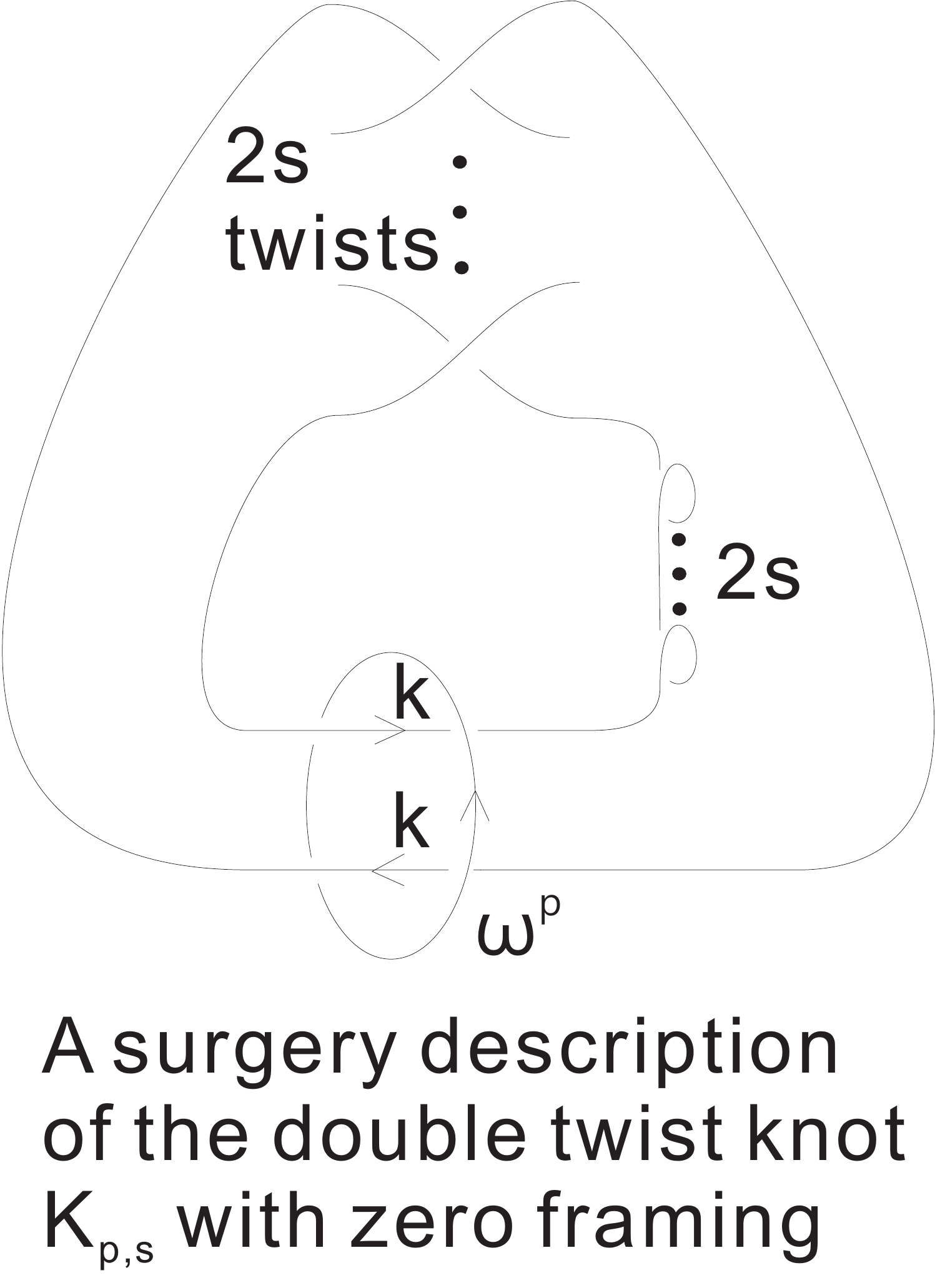}
\end{align*}

Let us calculate $\langle \mathcal{K}_{p,s}(H_N) \rangle$ by using
the above description for $\mathcal{K}_{p,s}$. Since $H_N$ has the
following presentation:
\begin{align*}
H_N=\sum_{k=0}^{N} \frac{ \{ N+k-1;a \}_{N-k} }{ \{ N-k\}! } R_k,
\end{align*}
one can insert this presentation along the above description instead
of $H_N$. Consider a pair $R_i \in \mathcal{K}_{p,s}(H_N)$ and $R_j$
from the expression $\omega_{N}^{p}$ (cf. formula
(\ref{formula-omegap})), the key observation is that this pair is a
cabling of the (twisted) Whitehead link, of which one component
pierces the other twice in the opposite direction each other.
Moreover, $R_i$ and $R_j$ are a linear combination of $D_{k,k}$ with
$k \leq i$ and $k \leq j$, respectively. According to Lemma
\ref{Coro:vanishi}, this pair vanishes if $i \neq j$. Therefore, we
can calculate $\langle \mathcal{K}_{p,s}(H_N) \rangle$ as follows:
\begin{align}
\langle \mathcal{K}_{p,s}(H_N)
\rangle=\sum_{k=0}^{N}\frac{\{N+k-1;a\}}{\{N-k\}!}t_{k,p}\raisebox{-40pt}{
\includegraphics[width=100
pt]{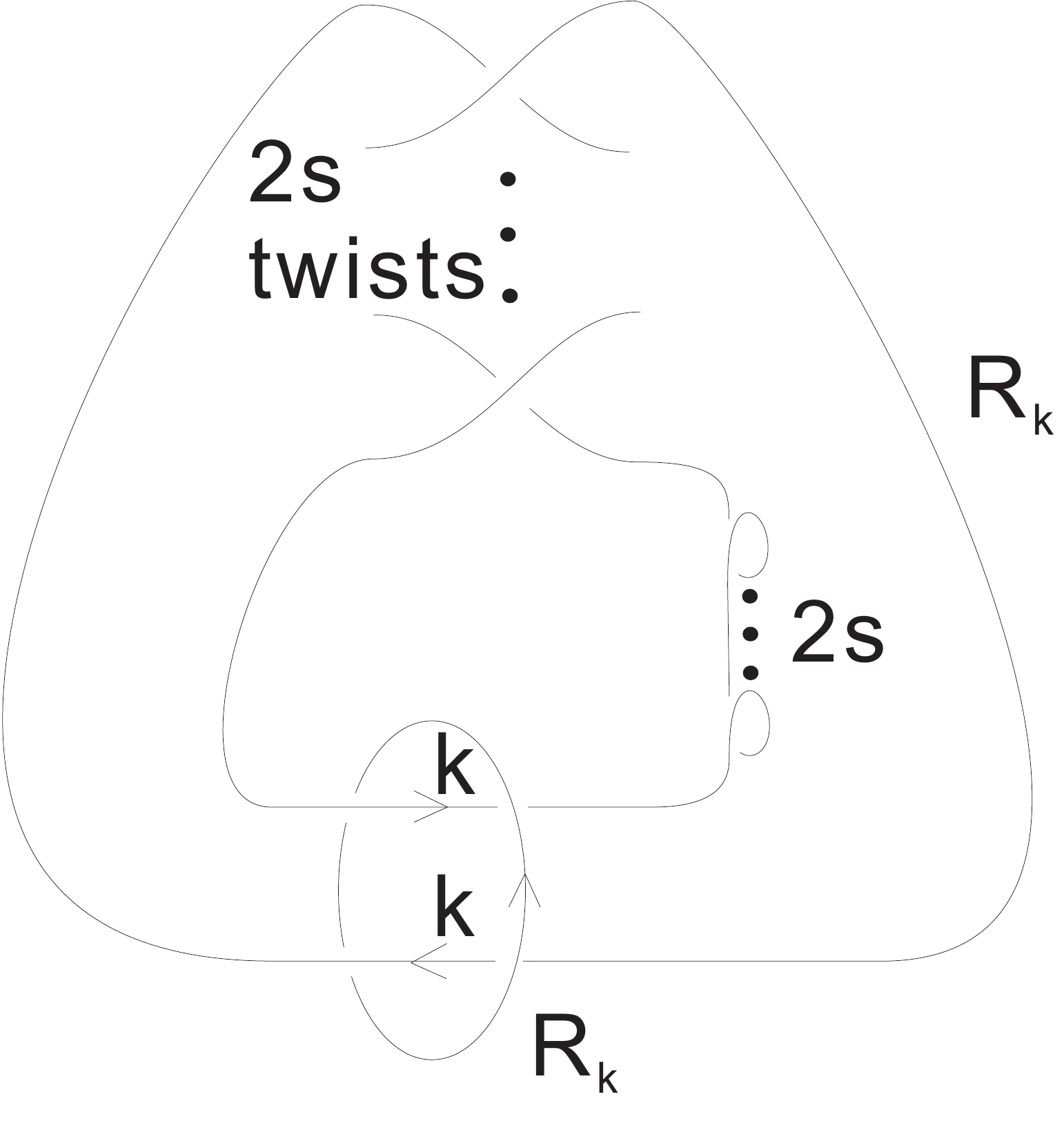}},
\end{align}
and since $R_k-H_k$ is a linear combination of $R_j$ with $j<k$, we
have
\begin{align*}
\raisebox{-40pt}{
\includegraphics[width=100
pt]{double-twist-2s-twist-RkRk.pdf}}&=\raisebox{-40pt}{
\includegraphics[width=100
pt]{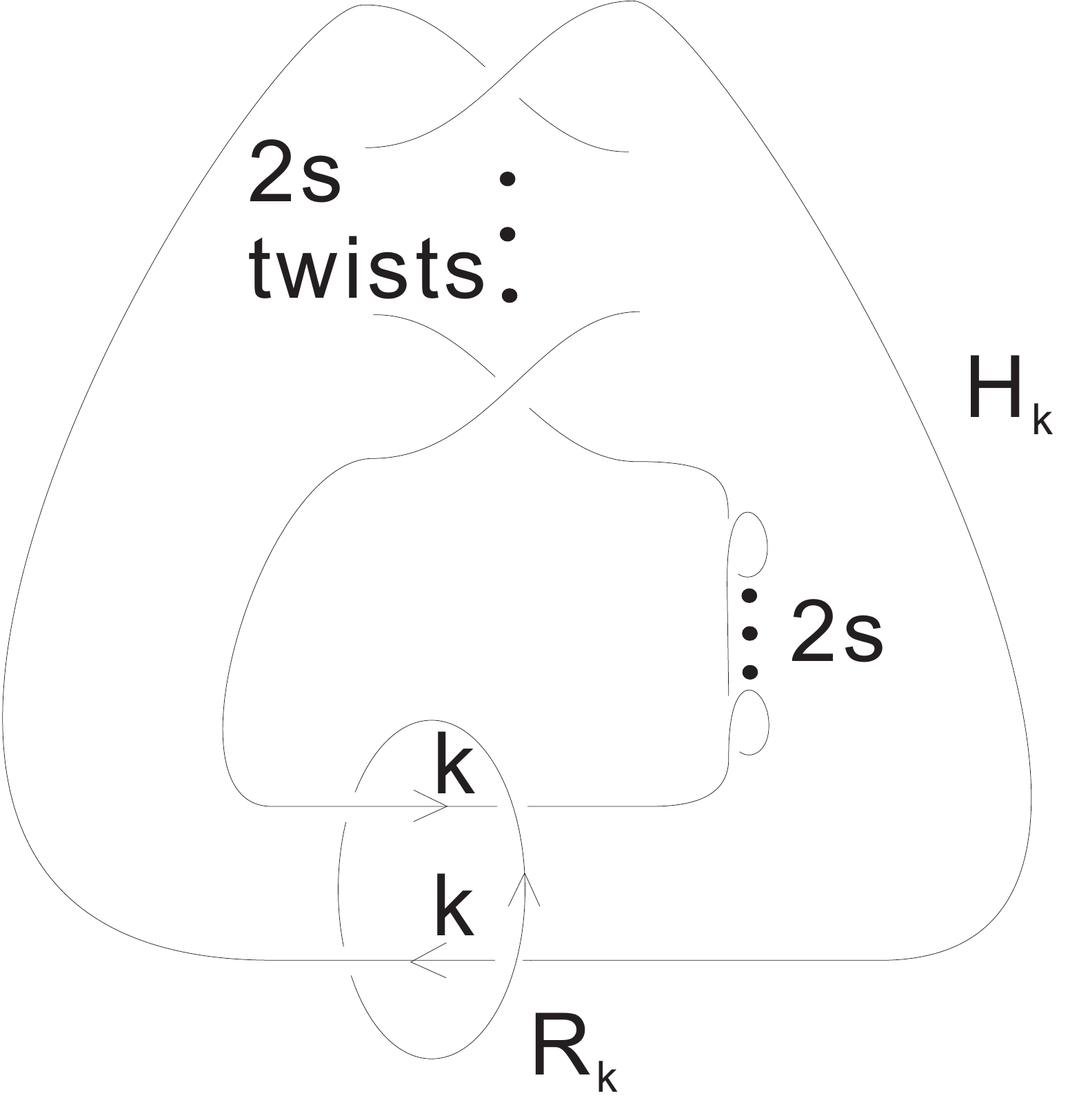}} \ \ \text{(by Lemma
\ref{Coro:vanishi})}\\\nonumber
&=(a^kq^{k(k-1)})^{2s}\raisebox{-40pt}{
\includegraphics[width=100
pt]{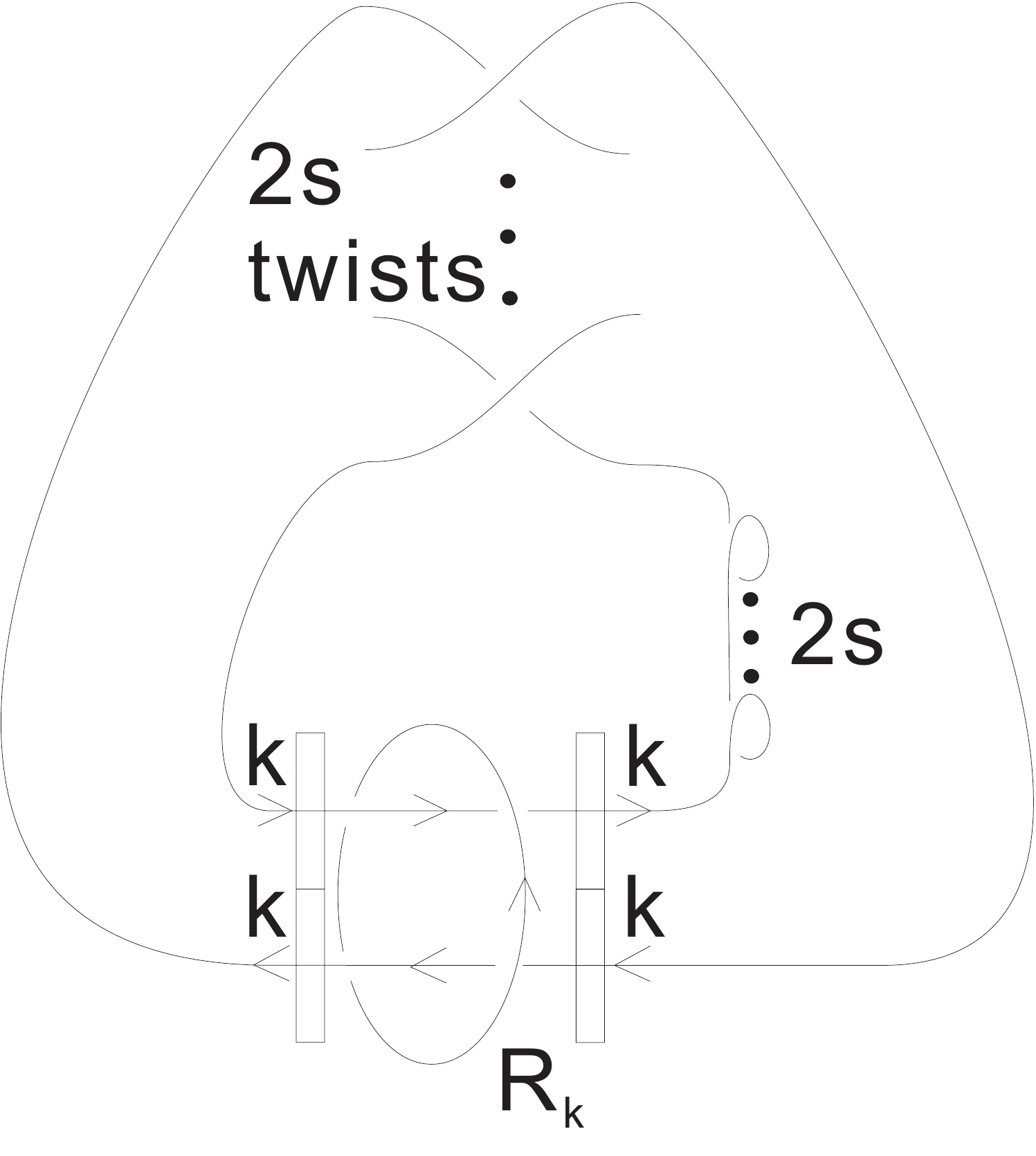}} \ \ \text{(By Lemma
\ref{lemma:ymn} and Lemma \ref{Coro:vanishi})}
\\\nonumber
&=(a^kq^{k(k-1)})^{2s}C_{k,k}^{(s)}\raisebox{-40pt}{
\includegraphics[width=120
pt]{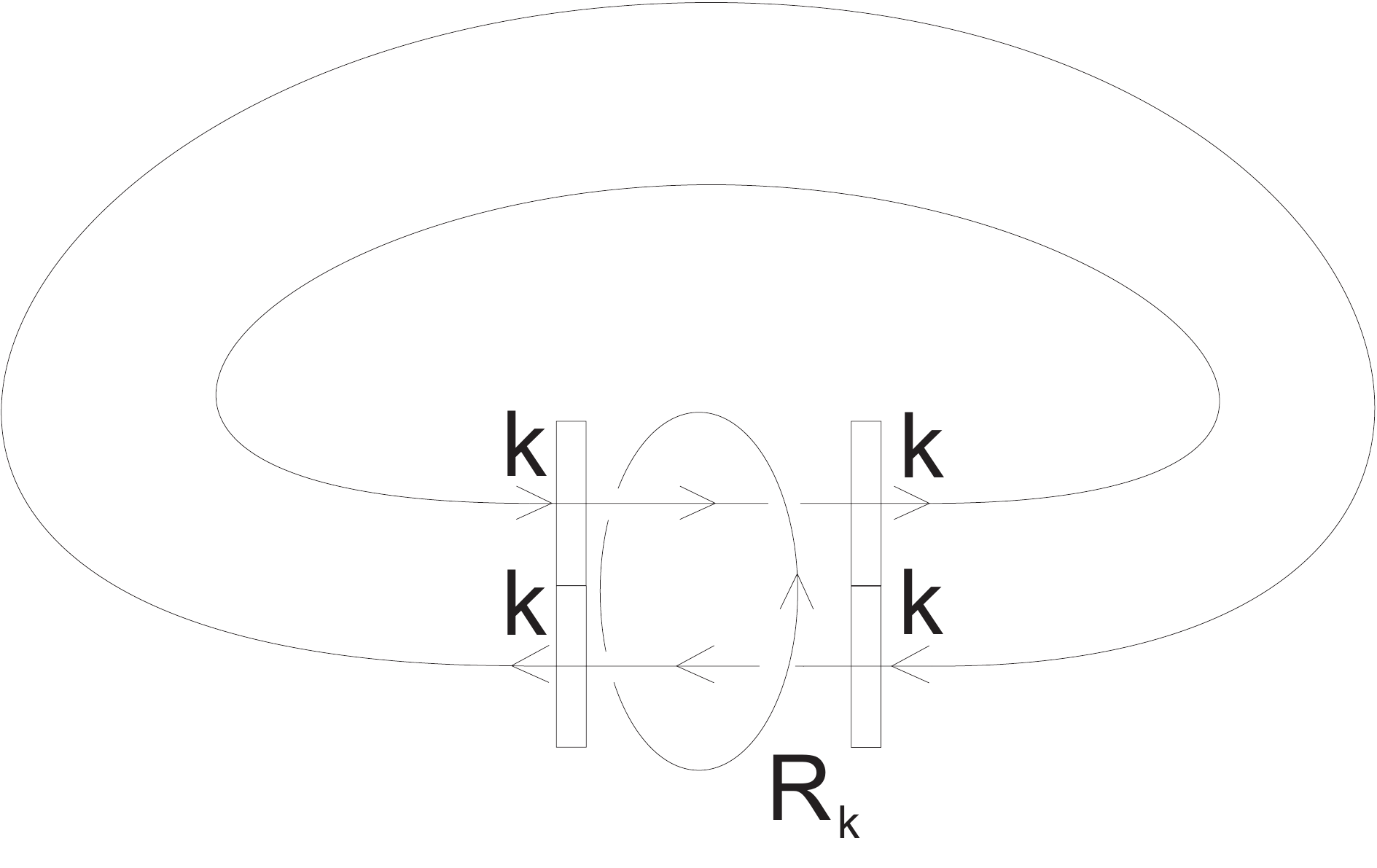}}\ \ \text{(By formula
(\ref{formula-Cnn}))}
\\\nonumber
&=(a^kq^{k(k-1)})^{2s}C_{k,k}^{(s)}\theta_{k,k}\langle
D_{k,k}\rangle \ \ \text{(By Lemma \ref{Coro:vanishi} and Lemma
\ref{lemma:Dmn})}\\\nonumber
&=(a^kq^{k(k-1)})^{2s}C_{k,k}^{(s)}\frac{\{2k-1\}_{2k}\{k-2;a\}_k}{\{k\}!}
\end{align*}

Therefore,
\begin{align} \label{formula-HN-Kps}
\mathscr{H}_{N}(\mathcal{K}_{p,s})&=\frac{\langle K_{p,s}(H_N)
\rangle}{\langle U(H_N)\rangle}\\\nonumber
&=\frac{\{N\}!}{\{N-1;a\}_N}\sum_{k=0}^{N}\frac{\{N+k-1;a\}_{N-k}}{\{N-k\}!}t_{k,p}
(a^kq^{k(k-1)})^{2s}C_{k,k}^{(2s)}\frac{\{2k-1\}_{2k}\{k-2;a\}_k}{\{k\}!}\\\nonumber
&=\sum_{k=0}^Nt_{k,p}(a^{k}q^{k(k-1)})^{2s}C_{k,k}^{(s)}\left[\begin{array}{@{\,}c@{\,}}N
\\ k\end{array} \right]\{N+k-1;a\}_k\{k-2;a\}_k.
\end{align}

By using formula (\ref{formula:Tnp=tnp}) and Lemma \ref{lemma:T=C},
we obtain
\begin{align}
t_{k,p}=(-1)^k(a^kq^{k(k-1)})^{2p}\frac{C_{k,k}^{(p)}}{\{k\}!^2}.
\end{align}
Substituting it back to formula (\ref{formula-HN-Kps}), we obtain
Theorem \ref{Theorem-main}.

As an application of our main Theorem \ref{Theorem-main}, we present
several concrete calculations for the twist knots with small
crossings.
\begin{example}
For the trefoil knot $3_1$ and figure-8 knot $4_1$ which are
$\mathcal{K}_{1,1}$ and $\mathcal{K}_{-1,1}$ respectively,

\begin{align} \label{formula-31}
\mathscr{H}_{N}(3_1;q,a)=\sum_{k=0}^{N}(-1)^ka^{2k}q^{k(k-1)}\left[
\begin{array}{@{\,}c@{\,}}N \\ k
\end{array} \right] \{ N+k-1;a \}_k \{k-2;a \}_k.
\end{align}
\begin{align} \label{formula-41}
\mathscr{H}_{N}(4_1;q,a)=\sum_{k=0}^{N}\left[
\begin{array}{@{\,}c@{\,}}N \\ k
\end{array} \right] \{ N+k-1;a \}_k \{k-2;a \}_k.
\end{align}
which are the formulae given by Corollary 5.2 in \cite{Kaw2}.
\end{example}

\begin{example}
For the $5_2$ knot $\mathcal{K}_{2,1}$, we have
\begin{align} \label{formula-52}
\mathscr{H}_{N}(5_2;q,a)=&\sum_{k=0}^{N}\left((-1)^ka^{4k}q^{3k(k-1)}\sum_{l=0}^{k}a^{-2l}q^{-3kl+l(l+2)}\left[
\begin{array}{@{\,}c@{\,}}k \\ l
\end{array} \right]\right)\\\nonumber&\cdot\left[
\begin{array}{@{\,}c@{\,}}N \\ k
\end{array} \right] \{ N+k-1;a \}_k \{k-2;a \}_k.
\end{align}
In particular, for $a=q^2$, we have $$\left[
\begin{array}{@{\,}c@{\,}}N \\ k
\end{array} \right]\{ N+k-1;a \}_k \{k-2;a
\}_k=\frac{\{N+k+1\}\{N+k\}\cdots\{N-k+1\}}{\{N+1\}}.$$

Then
\begin{align} \label{formula-52}
\mathscr{H}_{N}(5_2;q,q^2)=&\sum_{k=0}^{N}\left((-1)^kq^{3k^2+5k}\sum_{l=0}^{k}q^{-3kl+l(l-2)}\left[
\begin{array}{@{\,}c@{\,}}k \\ l
\end{array} \right]\right)\\\nonumber
&\cdot\frac{\{N+k+1\}\{N+k\}\cdots\{N-k+1\}}{\{N+1\}}.
\end{align}
which is equal to $J'_{K_2}(N+1)$ in \cite{Mas} ( cf. Theorem 5.1 in
\cite{Mas} for $p=2$, where the variable $a$ is just the variable
$q$ in our notation).
\end{example}

\begin{example}
For the $6_1$ knot $\mathcal{K}_{-2,1}$, we have
\begin{align} \label{formula-61}
\mathscr{H}_{N}(6_1;q,a)=&\sum_{k=0}^{N}\left(a^{-2k}q^{-2k(k-1)}\sum_{l=0}^{k}a^{2l}q^{3kl-l(l+2)}\left[
\begin{array}{@{\,}c@{\,}}k \\ l
\end{array} \right]\right)\\\nonumber&\cdot\left[
\begin{array}{@{\,}c@{\,}}N \\ k
\end{array} \right] \{ N+k-1;a \}_k \{k-2;a \}_k.
\end{align}
In particular, for $a=q^2$, we obtain
\begin{align} \label{formula-61}
\mathscr{H}_{N}(6_1;q,q^2)=&\sum_{k=0}^{N}\left(q^{-2k(k+1)}\sum_{l=0}^{k}q^{3kl-l(l-2)}\left[
\begin{array}{@{\,}c@{\,}}k \\ l
\end{array} \right]\right)\\\nonumber
&\cdot\frac{\{N+k+1\}\{N+k\}\cdots\{N-k+1\}}{\{N+1\}}.
\end{align}
which is equal to $J'_{K_{-2}}(N+1)$ in \cite{Mas} ( cf. Theorem 5.1
in \cite{Mas} for $p=-2$, where the variable $a$ is just the
variable $q$ in our notation).
\end{example}

\end{document}